\documentclass{article}
\usepackage{eurosym}
\usepackage{amssymb, amsmath}
\usepackage{verbatim}
\DeclareUnicodeCharacter{200B}{\hskip 0pt}
\newtheorem{theorem}{Theorem}[section]

\newtheorem{proposition}[theorem]{Proposition}
\newtheorem{corollary}[theorem]{Corollary}
\newtheorem{remark}[theorem]{Remark}

\numberwithin{equation}{section}

\newenvironment{proof}[1][Proof]{\noindent\textbf{#1.} }{\ \rule{0.5em}{0.5em}}

\begin{document}

\begin{center}
{\large \ Some properties of Padovan matrices and bi-periodic Padovan matrices}%
\begin{equation*}
\end{equation*}

\textbf{Diana Savin}%
\begin{equation*}
\end{equation*}
\end{center}

\textbf{Abstract.} {\small Let $\left(P_{n}\right)_{n\geq0}$ be the sequence of  bi-periodic Padovan numbers and let $\left(M_{p_{n}}\right)_{n\geq0}$ 
be the sequence  of bi-periodic Padovan matrices. In this article we study when these matrices are diagonalizable and we obtain a certain connection with the Lucas number sequence.
We also obtain some connections of these matrices with the generating matrix $Q$ for the Padovan numbers.} 
\medskip

\medskip

\begin{equation*}
\end{equation*}
 \section{Introduction and preliminaries}
\ \ \    The Fibonacci sequence $\left(F_n\right)_{n\ge 0}$ is defined by second-order linear recurrence relation 
$$F_{n}=F_{n-1}+F_{n-2},\ n\geq 2,$$
with the first two terms $F_{0}=0, F_{1}=1.$\\ 
   \ \ \  Closely related to the Fibonacci sequence was the Lucas sequence $\left(L_n\right)_{n\ge 0}$, also defined as a second-order linear recurrence
	$$L_{n}=L_{n-1}+L_{n-2},\ n\geq 2,$$
with the first two terms $L_{0}=2, L_{1}=1.$\\
If we consider the characteristic equation $t^{2}-t-1=0$ associated with the recurrence satisfied by the Fibonacci sequence and the Lucas sequence and denote by $t_{1}=\alpha$ and $t_{2}=\beta$ the solutions of this equation, it immediately follows:\\
- \textbf{Binet's formula for the Fibonacci numbers}: $F_{n}=\frac{\alpha^{n}-\beta^{n}}{\alpha-\beta}$\\
and\\
- \textbf{Binet's formula for the Lucas numbers}: $L_{n}=\alpha^{n}+\beta^{n}$.\\
   \ \ \  The Padovan sequence $\left(p_{n}\right)_{n\ge 0}$ is defined by a third-order linear recurrence relation 
              $$p_{n}=p_{n-2}+p_{n-3},\ n\geq 3,$$ 
with the first three terms $p_{0}=1,p_{1}=1$ and $p_{2}=1$. This number sequence was introduced by Richard Padovan in \cite{padovan}). Many mathematicians work with the shifted Padovan sequence, presenting as initial values ​​$p_{0}=1,p_{1}=0$ and $p_{2}=1$.\\
Many papers have been written about  the study of the properties of Fibonacci sequence, Lucas sequence,
Padovan sequence or others particular difference equations of degree $2$ or of degree $3$ with integer coefficients (see \cite{FlSa; 18}, 
\cite{Mu; 21}, \cite{Sa; 15}, \cite{Sa; 17}, \cite{Sa; 19}, \cite{Sa; 22}, \cite{Tan; Sa; Yil 23}, \cite{Ta; 21}).\\
In the paper \cite{Shannon; Anderson; Horadam; 06}, the authors considered the generating matrix for Padovan numbers: $Q=\left( 
\begin{array}{ccccc}
0 & 1 & 1  \\ 
1 & 0 & 0 \\ 
0 & 1 & 0 \\ 
\end{array}%
\right).$ They proved that  $Q^{n}=\left( 
\begin{array}{ccccc}
p_{n-2} & p_{n-1} & p_{n-3}  \\ 
p_{n-3} & p_{n-2} & p_{n-4} \\ 
p_{n-4} & p_{n-3} & p_{n-5} \\ 
\end{array}%
\right).$ \\
The matrix $Q$ has one real eigenvalue and two non-real complex eigenvalues, so it is not diagonalizable over $\mathbb{R}$ (see \cite{Vie; Catarino 22}).\\
     \ \ \    In the paper  \cite{diskaya}, Diskaya and Menken introduced the sequence $\left(P_n\right)_{n\ge 0}$ of  bi-periodic Padovan numbers like this

\[P_n =\begin{cases}a P_{n-2} + P_{n-3}, & \text{if } n \text{ is even}, \\b P_{n-2} + P_{n-3}, & \text{if } n \text{ is odd},\end{cases}\quad \text{for } n \geq 3,\] with the initial values $P_0 = 1,\ P_1 = 0,\ \text{and} \ P_2 = a.$
where where $a$ and $b$ are nonzero real numbers.\\
Also, in the paper  \cite{diskaya}, the authors introduced the sequence $\left(M_{p_{n}}\right)_{n\geq0}$ of  bi-periodic Padovan matrices like this

 \begin{equation*}
M_{p_{n}} =\begin{cases}a M_{p_{n-2}} + M_{p_{n-3}}, & \text{if } n \text{ is even}, \\b M_{p_{n-2}} + M_{p_{n-3}}, & \text{if } n \text{ is odd},\end{cases}\quad \text{for } n \geq 3,  \tag{1.1}
 \end{equation*}
where $a$ and $b$ are nonzero real numbers, with the initial terms 
$M_{p_{0}}=I_{3}, $
$M_{p_{1}}=\left( 
\begin{array}{ccccc}
0 & 1 & 0  \\ 
0 & 0 & 1 \\ 
1 & a & 0 \\ 
\end{array}%
\right), $
$M_{p_{2}}=\left( 
\begin{array}{ccccc}
0 & 0 & 1  \\ 
1 & a & 0 \\ 
0 & 1 & a \\ 
\end{array}%
\right).$\\
In this article we obtain some properties of matrices $Q$ and $M_{p_{n}},$ by studying whether they are diagonalizable or not.\\
In the article $\cite{Sa; 22},$ starting from the diagonalizable matrix $M=\left( 
\begin{array}{ccccc}
1 & 3 & 1  \\ 
1 & 0 & 0 \\ 
0 & 1 & 0 \\ 
\end{array}%
\right)$ we obtained a set of diagonalizable matrices, in connection with the matrix $M.$ In this article we find out for which real numbers $a$ and $b$ all matrices in the sequence $\left(M_{p_{n}}\right)_{n\geq0}$ are diagonalizable. For this we recall what it means that two matrices are simultaneously diagonalizable.\\
Let $A,$$B$$\in$$M_{n}\left(\mathbb{R}\right)$ be two diagonalizable matrices and $\lambda_{1}, \lambda_{2}, ...\lambda_{n}$ be the eigenvalues ​​of the matrix $A,$ and $\lambda^{'}_{1}, \lambda^{'}_{2}, ...\lambda^{'}_{n}$ be the eigenvalues ​​of the matrix $B.$ $A$ and $B$ are called simultaneously diagonalizable if there exists an invertible matrix $T\in$$M_{n}\left(\mathbb{R}\right)$ such that $A=T^{-1}\cdot \left( 
\begin{array}{ccccc}
\lambda_{1} & 0 & 0... & 0  \\ 
0 & \lambda_{2} & 0... & 0 \\ 
0 & ... & ..... & 0\\
0 & ... & ..... & 0\\
0 & ... &.. ..&\lambda_{n}\\ 
\end{array}%
\right)\cdot T$
and $B=T^{-1}\cdot \left( 
\begin{array}{ccccc}
\lambda^{'}_{1} & 0 & 0... & 0  \\ 
0 & \lambda^{'}_{2} & 0... & 0 \\ 
0 & ... & ..... & 0\\
0 & ... & ..... & 0\\
0 & ... &.. ..&\lambda^{'}_{n}\\ 
\end{array}%
\right)\cdot T.$\\
It is known that:\\
i) if $A,B$$\in$$M_{n}\left(\mathbb{R}\right)$ are simultaneously diagonalizable, then
$A\cdot B=B\cdot A$.\\
ii) The converse of i) is valid, provided that one of the matrices $A$ and $B$ has no multiple eigenvalues (see \cite{OlSha; 18}).
\begin{equation*}
\end{equation*}

\section{Results}
\begin{equation*}
\end{equation*}
\begin{proposition}
\label{prop2.1.} Let $\left(M_{p_{n}}\right)_{n\geq0}$ be the sequence of bi-periodic Padovan matrices defined in the introduction section. Then
$M_{p_{n}}\cdot  M_{p_{n+1}}=  M_{p_{n+1}}\cdot  M_{p_{n}}, \ \left(\forall\right) \  n\geq 0.  $

\end{proposition}
\begin{proof}
We proceed by mathematical induction over $n.$\\
For $n=0,$ since $M_{p_{0}}=I_{3},$ it immediately follows that 
$$M_{p_{0}}\cdot  M_{p_{1}}=M_{p_{1}}\cdot  M_{p_{0}}=M_{p_{1}}.$$
For $n=1,$ we obtain $M_{p_{1}}\cdot  M_{p_{2}}=$$\left( 
\begin{array}{ccccc}
1 & a & 0  \\ 
0 & 1 & a \\ 
a & a^{2} & 1 \\ 
\end{array}%
\right)=$$M_{p_{2}}\cdot  M_{p_{1}}.$\\
We suppose that the property is true for each $k\leq n$
: $M_{p_{k}}\cdot  M_{p_{k+1}}=  M_{p_{k+1}}\cdot  M_{p_{k}}.$ \\
We prove that the property is true for $n+1$
: $M_{p_{n+1}}\cdot  M_{p_{n+2}}=  M_{p_{n+2}}\cdot  M_{p_{n+1}}.$ \\
If $n$ is even, taking into account the recurrence relation of the  bi-periodic Padovan matrices, we have:
$$M_{p_{n+1}}\cdot  M_{p_{n+2}}=M_{p_{n+1}}\cdot \left(aM_{p_{n}} + M_{p_{n-1}}\right)=$$
$$aM_{p_{n+1}}\cdot M_{p_{n}}+ M_{p_{n+1}}\cdot M_{p_{n-1}}=$$
$$aM_{p_{n+1}}\cdot M_{p_{n}}+ \left(bM_{p_{n-1}} + M_{p_{n-2}}\right)\cdot M_{p_{n-1}}=$$
$$aM_{p_{n+1}}\cdot M_{p_{n}}+ bM^{2}_{p_{n-1}} + M_{p_{n-2}}\cdot M_{p_{n-1}}.$$
Applying the inductive hypothesis and the recurrence relation of the  bi-periodic Padovan matrices, we obtain:
$$M_{p_{n+1}}\cdot  M_{p_{n+2}}=aM_{p_{n}}\cdot M_{p_{n+1}}+ bM^{2}_{p_{n-1}} + M_{p_{n-1}}\cdot M_{p_{n-2}}=$$
$$aM_{p_{n}}\cdot M_{p_{n+1}}+ M_{p_{n-1}}\cdot \left(bM_{p_{n-1}} + M_{p_{n-2}}\right)=$$
$$aM_{p_{n}}\cdot M_{p_{n+1}}+ M_{p_{n-1}}\cdot M_{p_{n+1}}=$$
$$\left(aM_{p_{n}} + M_{p_{n-1}}\right)\cdot M_{p_{n+1}}=M_{p_{n+2}}\cdot M_{p_{n+1}}.$$
Analogously, we show that, if $n$ is odd, then\\ 
$M_{p_{n+1}}\cdot  M_{p_{n+2}}=  M_{p_{n+2}}\cdot  M_{p_{n+1}}.$
\end{proof}

\smallskip

We now study for which real numbers $a,$ the Padovan bi-periodic matrices are diagonalizable.\\
First, we consider $a=2$ and obtain the following result
\begin{proposition}
\label{prop2.2.} Let $\left(M_{p_{n}}\right)_{n\geq0}$ be the sequence of bi-periodic Padovan matrices with $a=2$ and $b$ is an arbitrary real number. Then, we have:\\
i) the matrices $M_{p_{n}}$ are diagonalizable, for any $n$ positive integer;\\
ii) if $\lambda_{1n},$ $\lambda_{2n}$ and $\lambda_{3n}$
are the eigenvalues of the matrix $M_{p_{n}},$ for each $n$ positive integer, then 
\[\lambda_{in} =\begin{cases}2\lambda_{in-2} + \lambda_{in-3}, & \text{if } n \text{ is even}, \\b\lambda_{in-2} + \lambda_{in-3}, & \text{if } n \text{ is odd},\end{cases}\quad \text{for } n \geq 3,\ i=\overline{1,3}.\]

\end{proposition}
  
\begin{proof}
$M_{p_{0}}=I_{3}$  is diagonalizable and $\lambda_{10}=\lambda_{20}=\lambda_{30}=1$.\\ 
$M_{p_{1}}=\left( 
\begin{array}{ccccc}
0 & 1 & 0  \\ 
0 & 0 & 1 \\ 
1 & 2 & 0 \\ 
\end{array}%
\right) .$\\
 Solving the characteristic equation $det\left(M_{p_{1}}-\lambda\cdot I_{3}\right)=0$, we obtain the eigenvalues:
$\lambda_{11}=-1,$  $\lambda_{21}=\frac{1+\sqrt{5}}{2},$ $\lambda_{31}=\frac{1-\sqrt{5}}{2}.$ Since $M_{p_{1}}$
has $ 3$ real  distinct eigenvalues, it results immediately that the matrix $M_{p_{1}}$ is diagonalizable.\\ \\
$M_{p_{2}}=\left( 
\begin{array}{ccccc}
0 & 0 & 1  \\ 
1 & 2 & 0 \\ 
0 & 1 & 2 \\ 
\end{array}%
\right).$ Solving the equation $det\left(M_{p_{2}}-\lambda\cdot I_{3}\right)=0$, we obtain the eigenvalues: $\lambda_{12}=1,$  $\lambda_{22}=\frac{3+\sqrt{5}}{2},$ $\lambda_{32}=\frac{3-\sqrt{5}}{2}.$ Since $M_{p_{2}}$ has $ 3$ real  distinct eigenvalues, it results immediately that the matrix $M_{p_{2}}$ is diagonalizable.\\
But, applying to Proposition \ref{prop2.1.}, $M_{p_{1}}\cdot M_{p_{2}}=M_{p_{2}}\cdot M_{p_{1}}$ and            according to 8.3.26, p.429 (\cite{OlSha; 18}) we obtain that the matrices $M_{p_{1}}$ and $M_{p_{2}}$ are simultaneously diagonalizable, so there exists  an invertible matrix $T$$\in 
\mathcal{M}_{3}\left( \mathbb{R}\right) $ such that: 
\begin{equation}
M_{p_{1}}= T^{-1}\cdot \left( 
\begin{array}{ccccc}
\lambda_{11} & 0 & 0  \\ 
0 & \lambda_{21} & 0 \\ 
0 & 0 & \lambda_{31} \\ 
\end{array}%
\right)\cdot T=
 T^{-1}\cdot \left( 
\begin{array}{ccccc}
-1 & 0 & 0  \\ 
0 & \frac{1+\sqrt{5}}{2} 
 & 0 \\ 
0 & 0 & \frac{1-\sqrt{5}}{2} 
 \\ 
\end{array}%
\right)\cdot T. 
\tag{2.1.}
\end{equation}
 and 
\begin{equation}
M_{p_{2}}= T^{-1}\cdot \left( 
\begin{array}{ccccc}
\lambda_{12} & 0 & 0  \\ 
0 & \lambda_{22} & 0 \\ 
0 & 0 & \lambda_{32} \\ 
\end{array}%
\right)\cdot T
=
 T^{-1}\cdot \left( 
\begin{array}{ccccc}
1 & 0 & 0  \\ 
0 & \frac{3+\sqrt{5}}{2} 
 & 0 \\ 
0 & 0 & \frac{3-\sqrt{5}}{2} 
 \\ 
\end{array}%
\right)\cdot T.  \tag{2.2.}
\end{equation}
From the relations (2.1),  (2.2) and the recurrence relation of the  bi-periodic Padovan matrices,  it follows that
\begin{equation}
M_{p_{3}}= T^{-1}\cdot \left( 
\begin{array}{ccccc}
b\lambda_{11}+\lambda_{10} & 0 & 0  \\ 
0 & b\lambda_{21}+\lambda_{20} & 0 \\ 
0 & 0 & b\lambda_{31}+\lambda_{30} \\ 
\end{array}% 
\right)\cdot T. \tag{2.3.}
\end{equation}
and 
\begin{equation}
M_{p_{4}}= T^{-1}\cdot \left( 
\begin{array}{ccccc}
2\lambda_{12}+\lambda_{11} & 0 & 0  \\ 
0 & 2\lambda_{22}+\lambda_{21} & 0 \\ 
0 & 0 & 2\lambda_{32}+\lambda_{31} \\ 
\end{array}%
\right)\cdot T. \tag{2.4.}
\end{equation}
From (2.3.), it results that $\lambda_{i3}=b\lambda_{i1}+\lambda_{i0},$ for $i=\overline{1,3}.$\\
From (2.4.), it results that $\lambda_{i4}=2\lambda_{i2}+\lambda_{i1},$ for $i=\overline{1,3}.$\\
We prove by mathematical induction over $n\geq1$ the statement\\
$P\left(n\right):$ 
$M_{p_{n}}= T^{-1}\cdot \left( 
\begin{array}{ccccc}
\lambda_{1n} & 0 & 0  \\ 
0 & \lambda_{2n} & 0 \\ 
0 & 0 & \lambda_{3n} \\ 
\end{array}%
\right)\cdot T$\\
and 
\[\lambda_{in} =\begin{cases}2\lambda_{in-2} + \lambda_{in-3}, & \text{if } n \text{ is even}, \\b\lambda_{in-2} + \lambda_{in-3}, & \text{if } n \text{ is odd},\end{cases}\quad \text{for } n \geq 3,\ i=\overline{1,3}.\]
From (2.3.), it results that $P\left(3\right)$ is true.\\
From (2.4.), it results that $P\left(4\right)$ is true.\\
We assume that $P\left(l\right)$ is true for each $l\leq n$ and we show that $P\left(n+1\right)$ is true.\\
We consider two cases.\\
Case 1: when $n$ is odd. Applying (1.1.) and the inductive hypothesis, it follows that
$$M_{p_{n+1}}=2M_{p_{n-1}}+M_{p_{n-2}}= $$
\begin{equation}
=T^{-1}\cdot \left( 
\begin{array}{ccccc}
2\lambda_{1n-1}+\lambda_{1n-2} & 0 & 0  \\ 
0 & 2\lambda_{2n-1}+\lambda_{2n-2} & 0 \\ 
0 & 0 & 2\lambda_{3n-1}+\lambda_{3n-2} \\ 
\end{array}%
\right)\cdot T. \tag{2.5.}
\end{equation}
Case 2: when $n$ is even. Applying (1.1.) and the inductive hypothesis, it follows that
$$M_{p_{n+1}}=bM_{p_{n-1}}+M_{p_{n-2}}= $$
\begin{equation}
=T^{-1}\cdot \left( 
\begin{array}{ccccc}
b\lambda_{1n-1}+\lambda_{1n-2} & 0 & 0  \\ 
0 & b\lambda_{2n-1}+\lambda_{2n-2} & 0 \\ 
0 & 0 & b\lambda_{3n-1}+\lambda_{3n-2} \\ 
\end{array}%
\right)\cdot T. \tag{2.6.}
\end{equation}
From (2.5) and (2.6), it results that $P\left(n+1\right)$ is true.\\
We obtain that $P\left(n\right)$ is true for each positive integer $n\geq 3.$ From here it follows that
the matrices $M_{p_{n}}$ are diagonalizable, for any $n$ positive integer and its eigenvalues ​​are
\[\lambda_{in} =\begin{cases}2\lambda_{in-2} + \lambda_{in-3}, & \text{if } n \text{ is even}, \\b\lambda_{in-2} + \lambda_{in-3}, & \text{if } n \text{ is odd},\end{cases}\quad \text{for } n \geq 3,\ i=\overline{1,3}.\]
\end{proof}

\begin{corollary}
\label{cor2.3}
Let $\left(M_{p_{n}}\right)_{n\geq0}$ be the sequence of bi-periodic Padovan matrices with $a=2$ and $b$ is an arbitrary real number. Then
 $$ M_{p_{n}}\cdot M_{p_{m}}= M_{p_{m}}\cdot M_{p_{n}}, \ \left(\forall\right)\  n,m\in \mathbb{N}.$$

\end{corollary}

\begin{proof}
From the statement $P\left(n\right)$ proved in Proposition \ref{prop2.2.} it follows that the matrices $M_{p_{n}}$ and $M_{p_{m}}$ are simultaneously diagonalizable, for any positive integers $n$ and $m$ and from this we immediately obtain that 
 $$M_{p_{n}}\cdot M_{p_{m}}= M_{p_{m}}\cdot M_{p_{n}}, \ \left(\forall\right)\  n,m\in \mathbb{N}.$$
\end{proof}

For some values ​​of $a\geq2$ we used Magma Computer Algebra System software and we saw that the matrix $M_{p_{1}}$ is diagonalizable. We obtain the following result.
\begin{proposition}
\label{prop2.4.}
The matrix $M_{p_{1}}$ is diagonalizable for any $a\geq2.$
\end{proposition}
\begin{proof}
For $a=2,$ $M_{p_{1}}$ is diagonalizable, according to Proposition \ref{prop2.2.}.\\
Now, we consider the case $a>2.$\\
The characteristic equation 
$$det\left(M_{p_{1}}-\lambda\cdot I_{3}\right)=0\Leftrightarrow
\begin{vmatrix}
-\lambda & 1 & 0  \\ 
0 & -\lambda & 1 \\ 
1 & a & -\lambda \\ 

\end{vmatrix}=0\Leftrightarrow \lambda^{3}-a\lambda-1=0.$$
We consider the map $f:\mathbb{R}\longmapsto \mathbb{R},$ $f\left(x\right)=x^{3}-ax-1.$ 
The map $f$ is continuous on $\mathbb{R}.$ We calculate $f\left(-a\right)=-a^{3}+a^{2}-1=a^{2}\cdot \left(1-a\right)-1<0,$ $f\left(-1\right)=-2+a>0,$
$f\left(0\right)=-1<0,$ $f\left(a\right)=a^{3}-a^{2}-1=a^{2}\cdot \left(a-1\right)-1 >0,$ $\left(\forall\right)$ $a>2,$ so, the equation $x^{3}-ax-1=0$ has $3$ distinct real solutions:
$x_{1}\in\left(-a,-1\right),$ $x_{2}\in\left(-1,0\right),$ $x_{3}\in\left(0,a\right).$ So, the matrix $M_{p_{1}}$has $3$ distinct real eigenvalues. It follows from this that the matrix $M_{p_{1}}$ is diagonalizable.
\end{proof}

\begin{proposition}
\label{prop2.5.} Let $\left(M_{p_{n}}\right)_{n\geq0}$ be the sequence of bi-periodic Padovan matrices with $a>2$ and $b$ is an arbitrary real number. Then, we have:\\
i) the matrices $M_{p_{n}}$ are diagonalizable, for any $n$ positive integer;\\
ii) if $\lambda_{1n},$ $\lambda_{2n}$ and $\lambda_{3n}$
are the eigenvalues of the matrix $M_{p_{n}},$ for each $n$ positive integer, then 
\[\lambda_{in} =\begin{cases}a\lambda_{in-2} + \lambda_{in-3}, & \text{if } n \text{ is even}, \\b\lambda_{in-2} + \lambda_{in-3}, & \text{if } n \text{ is odd},\end{cases}\quad \text{for } n \geq 3,\ i=\overline{1,3}.\]

\end{proposition}
  
\begin{proof}
Taking into account the result of Proposition \ref{prop2.4.}, the proof of Proposition \ref{prop2.5.} is similar to the proof of Proposition \ref{prop2.2.}
\end{proof}\\

\smallskip

By immediate calculations we obtain the following remark on the matrix $M_{p_{1}}.$
\begin{remark}
\label{remark2.6.}
$M^{2}_{p_{1}}=M_{p_{2}},$ $M^{4}_{p_{1}}=M_{p_{4}},$ for any real number $a.$
\end{remark}
In proving our next results we used the following remark.
\begin{remark}
\label{remark2.7.} (\cite{OlSha; 18}).
i) $Tr\left(T\cdot A\cdot T^{-1}\right)=Tr\left(A\right),$ for any $A$$\in$$M_{n}\left(\mathbb{C}\right)$ and for any invertible matrix $T$$\in$$M_{n}\left(\mathbb{C}\right).$\\
ii) Let $n$ be a positive integer, let $K$ be a commutative field and let $A$$\in$$M_{n}\left(K\right)$. If the matrix $A$ has the eigenvalues ​​$\lambda_{1}, \lambda_{2},...,\lambda_{n}$, then the matrix $A^{l}$ has the eigenvalues ​​$\lambda^{l}_{1}, \lambda^{l}_{2},...,\lambda^{l}_{n}$, where $l\in$$\mathbb{N}^{*}.$
\end{remark}
For $a=2,$ we obtain the following result about matrices $M^{n}_{p_{1}},$ with $n$ positive integers.
\begin{proposition}
\label{prop2.8.} 
Let $\left(L_{n}\right)_{n\geq0}$ be the sequence of Lucas numbers and
let $\left(M_{p_{n}}\right)_{n\geq0}$ be the sequence of bi-periodic Padovan matrices with $a=2$ and $b$ is an arbitrary real number. Then $Tr\left(M^{n}_{p_{1}}\right)= L_{n}+\left(-1\right)^{n}.$

\end{proposition}
\begin{proof} If, we denote $\alpha=\frac{1+\sqrt{5}}{2}$ and $\beta=\frac{1-\sqrt{5}}{2},$ according to Binet's formula for Lucas numbers, we have $L_{n}=\alpha^{n}+\beta^{n}.$\\
From (2.1.), there exists  an invertible matrix $T$$\in 
\mathcal{M}_{3}\left( \mathbb{R}\right) $ such that: 
$$M_{p_{1}}=
T^{-1}\cdot \left( 
\begin{array}{ccccc}
-1 & 0 & 0  \\ 
0 & \alpha  
 & 0 \\ 
0 & 0 & \beta
 \\ 
\end{array}%
\right)\cdot T.$$
From here, applying Remark \ref{remark2.7.} ii) it follows that 
$$M^{n}_{p_{1}}=T^{-1}\cdot \left( 
\begin{array}{ccccc}
\left(-1\right)^{n} & 0 & 0  \\ 
0 & \alpha^{n}  
 & 0 \\ 
0 & 0 & \beta^{n}
 \\ 
\end{array}%
\right)\cdot T,$$ and applying Remark \ref{remark2.7.} i) we obtain $Tr\left(M^{n}_{p_{1}}\right)= L_{n}+\left(-1\right)^{n}.$
\end{proof}

\medskip

If we denote by $B=Q\cdot M_{p_{1}},$ we quickly obtain that the eigenvalues ​​of the matrix $B$ are $\lambda_{1}=\lambda_{2}=\lambda_{3}=1$, but $B$ is not diagonalizable.
\begin{proposition}
\label{prop2.9.} i) $B^{n}=\left( 
\begin{array}{ccccc}
1 & na & n  \\ 
0 & 1  & 0 \\ 
0 & 0 & 1
 \\ 
\end{array}\right),$ 
for any $n\in$$\mathbb{N}^{*}.$\\
ii) If $a$ is a real number, the set of matrices $G=\left\{\left( 
\begin{array}{ccccc}
1 & na & n  \\ 
0 & 1  & 0 \\ 
0 & 0 & 1
 \\ 
\end{array}%
\right)| n\in \mathbb{Z}\right\}$ forms a group with respect to matrix multiplication.

\end{proposition}
\begin{proof} i) $B=Q\cdot M_{p_{1}}=\left( 
\begin{array}{ccccc}
1 & a & 1  \\ 
0 & 1  & 0 \\ 
0 & 0 & 1
 \\ 
\end{array}%
\right).$
It is quickly shown by induction over $n\in \mathbb{N}^{*}$ that $B^{n}=\left( 
\begin{array}{ccccc}
1 & na & n  \\ 
0 & 1  & 0 \\ 
0 & 0 & 1
 \\ 
\end{array}%
\right).$\\
ii) The matrix $B$ is invertible and $B^{-1}=\left( 
\begin{array}{ccccc}
1 & -a & -1  \\ 
0 & 1  & 0 \\ 
0 & 0 & 1
 \\ 
\end{array}%
\right)$. It is easily proven by induction over $n\in \mathbb{N}^{*}$ that $\left(B^{-1}\right)^{n}=\left( 
\begin{array}{ccccc}
1 & -na & -n  \\ 
0 & 1  & 0 \\ 
0 & 0 & 1
 \\ 
\end{array}%
\right)$.\\
Using these, it is easy to show that $G$ is a subgroup of the group $\left(GL_{n}\left(\mathbb{R}\right), \cdot\right)$ of invertible matrices of order $n,$ with real number elements.
\end{proof}

\begin{equation*}
\end{equation*}

\begin{equation*}
\end{equation*}

\begin{equation*}
\end{equation*}%
\qquad

Diana SAVIN

{\small Faculty of Mathematics and Computer Science, }

{\small  Transilvania University of Bra\c{s}ov,}

{\small Iuliu Maniu street 50, Braşov 500091, Rom\^{a}nia, }

{\small https://www.unitbv.ro/en/contact/search-in-the-unitbv-community/4292-diana-savin.html}

{\small e-mail: \ diana.savin@unitbv.ro, \ dianet72@yahoo.com}

\end{document}